\documentclass[11pt,letterpaper]{amsart}
\usepackage{graphicx, psfrag, amscd, amssymb}
\usepackage[vcentermath]{youngtab} 

\newtheorem{lem}{Lemma}[section]
\newtheorem{thm}[lem]{Theorem}
\newtheorem{pro}[lem]{Proposition}
\newtheorem{cor}[lem]{Corollary}
\newtheorem{exa}[lem]{Example}
\newtheorem{con}[lem]{Conjecture}

\newcommand{\ZZ}{{\mathbb{Z}}}
\newcommand{\RR}{{\mathbb{R}}}

\newcommand{\M}{{\mathcal{M}}}
\newcommand{\T}{{\mathcal{T}}}

\begin{document}

\title{Standard Young Tableaux and Colored Motzkin Paths}

\author{Sen-Peng Eu}
\address{Department of Applied Mathematics, National University of Kaohsiung, and Chinese Air Force Academy, Kaohsiung 811, Taiwan, ROC}
\email{speu@nuk.edu.tw}

\author{Tung-Shan Fu}
\address{Mathematics Faculty, National Pingtung Institute of Commerce, Pingtung 900, Taiwan, ROC}
\email{tsfu@npic.edu.tw}

\author{Justin T. Hou}
\address{Department of Electrical Engineering, National Taiwan University, Taipei 106, Taiwan, ROC}
\email{b99901160@ntu.edu.tw}

\author{Te-Wei Hsu}
\address{Department of Physics, National Taiwan University, Taipei 106, Taiwan, ROC}
\email{b99202053@ntu.edu.tw}

\thanks{Research partially supported by NSC grants 101-2115-M-390-004-MY3 (S.-P. Eu) and
 101-2115-M-251-001 (T.-S. Fu).}

\begin{abstract}
In this paper, we propose a notion of colored Motzkin paths and
establish a bijection between the $n$-cell standard Young tableaux (SYT) of bounded height and the colored Motzkin paths of length
$n$. This
result not only gives a lattice path interpretation of the standard
Young tableaux but also reveals an unexpected
intrinsic relation between the set of SYTs with at most $2d+1$ rows and the set of SYTs with at most
$2d$ rows.
\end{abstract}

\maketitle

\section{Introduction}

\baselineskip=15pt

\subsection{Standard Young tableaux with bounded height}
Let $\lambda=(\lambda_1,\lambda_2,\dots)$ be a partition of $n$,
where $\lambda_1\ge\lambda_2\ge\dots\ge 0$. The \emph{shape} of
$\lambda$ is a left-justified array of cells with $\lambda_i$
cells in the $i$th row. A \emph{standard Young tableaux} (SYT) of
shape $\lambda$ is a filling of the array with $\{1,2,\dots ,n\}$
such that every row and column is increasing. The concept of SYT
is fundamental in combinatorics and representation theory. For
example, the irreducible representation of the symmetric group
$\mathfrak{S}_n$ is indexed by the SYTs of shape $\lambda$, whose
dimension can be enumerated by the celebrated hook-length
formulae~\cite{Frame_54}.

The enumeration of SYTs of bounded height is a difficult problem.
Let $\T_n^{(d)}$ denote the set of SYTs with $n$ cells and at most $d$ rows.
The nontrivial explicit formulae are only known for $d=2,3,4,5$. In
terms of representation theory, Regev \cite{Regev-1} proved
that
\[
|\T_n^{(2)}|={n\choose \lfloor \frac{n}{2}\rfloor} \qquad \mbox{and} \qquad |\T_n^{(3)}|=\sum_{i\ge 0}\frac{1}{i+1}{n\choose 2i}{2i\choose i}.
\]
Gouyou-Beauchamps \cite{Gouyou} derived combinatorially that
\[
|\T_n^{(4)}|=c_{\lfloor \frac{n+1}{2}\rfloor}c_{\lceil
\frac{n+1}{2}\rceil} \qquad \mbox{and} \qquad
|\T_n^{(5)}|=6\sum_{i=0}^{\lfloor \frac{n}{2}\rfloor}{n\choose
2i}\frac{(2i+2)!c_i}{(i+2)!(i+3)!},
\] where
$c_n=\frac{1}{n+1}{2n\choose n}$ is the $n$th Catalan
number.

Though explicit formulae are elusive for $d\ge 6$, some generating functions for $|\T_n^{(d)}|$
are known. Gessel \cite{Gessel} proved that the exponential
generating functions for $|\T_n^{(2d+1)}|$ and $|\T_n^{(2d)}|$ are
\[
e^z\det[I_{i-j}(2z)-I_{i+j}(2z)]^{d}_{i,j=1} \qquad \mbox{and} \qquad \det[I_{i-j}(2z)+I_{i+j-1}(2z)]^{d}_{i,j=1}
\]
respectively, where $I_m(2z)=\sum_{j\ge 0}
\frac{z^{m+2j}}{j!(m+j)!}$ is the hyperbolic Bessel function of the
first kind of order $m$; see also \cite[section 5]{Stanley-paper}
for more information.

\subsection{Motzkin paths}
Note that $|\T_n^{(3)}|$ is a Motzkin number, which counts
the number of Motzkin paths of length $n$ . A \emph{Motzkin path} of
length $n$ is a lattice path in the plane $\RR^2$ from the origin to
the point $(n,0)$ never going below
the $x$-axis, with the \emph{up steps} $(1,1)$, \emph{down
steps} $(1,-1)$ and \emph{level steps} $(1,0)$ .

In previous work \cite{Eu}, the first author established a bijection
between the Motzkin paths of length $n$ and the $n$-cell SYTs with at most three rows, which offers a
combinatorial proof of Regev's result.
Moreover, an extension of this connection was also considered in \cite{Eu}. Let
$\{\epsilon_1,\dots,\epsilon_{d+1}\}$ denote the standard basis of
$\RR^{d+1}$.  Combining analytic work of Grabiner-Magyar
\cite{Grab-Mag} and Gessel \cite{Gessel}, the set $\T_n^{(2d+1)}$ is related to the following lattice path
enumeration in $\RR^{d+1}$, which is equivalent to enumeration of
Zeilberger's lazy walks (defined below).

\begin{thm} \label{thm:1-1}
Using the following steps as the admissible unit steps
\[
\{\epsilon_1\}\cup\{\epsilon_1+\epsilon_2,\epsilon_1-\epsilon_2\}\cup\{\epsilon_1-\epsilon_i+\epsilon_{i+1},
\epsilon_1+\epsilon_i-\epsilon_{i+1}:2\le i\le d\},
\]
the number of paths of
length $n$ from the origin to the point $(n,0,\dots,0)$ staying
within the nonnegative octant equals the number of $n$-cell SYTs with at most
$2d+1$ rows.
\end{thm}

For SYTs with at most $2d$
rows, the following refined enumeration was conjectured in \cite{Eu}.

\begin{con} \label{cor:1-2}
If the steps $\epsilon_1$ are confined to the hyperplane spanned by $\{\epsilon_1,\dots,\epsilon_d\}$,
then the number of paths equals the number of $n$-cell SYTs with at most $2d$
rows.
\end{con}

For $d=1$, this conjecture is supported by the fact that the number
of Motzkin paths of length $n$ with all level steps on the $x$-axis,
which is the central binomial number ${n\choose \lfloor \frac{n}{2}
\rfloor}$, equals the number of $n$-cell SYTs with at most two rows.

\medskip

Motivated by this study, in the present work, we propose a notion of
colored Motzkin paths in terms of shuffle of multiple parenthesis
systems and give a combinatorial proof of Theorem \ref{thm:1-1} and
Conjecture \ref{cor:1-2}.

Our results not only establish a bijection between SYTs and colored Motzkin paths, but also reveal a link between
the two sets $\T_n^{(2d+1)}$ and $\T_n^{(2d)}$ via a property of level steps of the paths (Theorem \ref{thm:main-bijection}).

\subsection{The $d$-Motzkin words and $d$-Motzkin paths}
A \emph{word} $\pi$ on a set $X$ (called the \emph{alphabet}) is a
finite sequence of letters in $X$. The \emph{length} of $\pi$,
denoted by $|\pi|$, is the number of letters in $\pi$. For a
letter $x\in X$, the number of occurrences of $x$ in $\pi$ is
denoted by $|\pi|_x$. A word $\pi$ on the alphabet $\{x, \overline{x}\}$
is called a \emph{parenthesis system} if $|\pi|_x=|\pi|_{\overline{x}}$ and for every prefix
$\mu$ of $\pi$, $|\mu|_x\ge |\mu|_{\overline{x}}$.

\smallskip
Consider the alphabet $X=\{L\}\cup\{U^{(i)},D^{(i)}: 1\le i\le d\}$.
A word $\pi$ on $X$ is called a \emph{$d$-Motzkin word} if it satisfies the following two conditions:
\begin{enumerate}
\item[(M1)] $|\pi|_{U^{(i)}}=|\pi|_{D^{(i)}}$, for all $1\le i\le d$;
\item[(M2)] for every prefix $\mu$,
$|\mu|_{U^{(1)}}-|\mu|_{D^{(1)}}\ge |\mu|_{U^{(2)}}-|\mu|_{D^{(2)}}\ge\cdots\ge |\mu|_{U^{(d)}}-|\mu|_{D^{(d)}}\ge 0$.
\end{enumerate}
Such a word can be visualized in terms of Motzkin paths with
colored steps. We associate the letter $U^{(i)}$ (respectively,
$D^{(i)}$) to an up step $(1,1)$ (respectively, down step
$(1,-1)$) with the $i$th color, for $1\le i\le d$, and associate
the letter $L$ to a level step $(1,0)$ without any color. Then a
$d$-Motzkin word of length $n$ is associated to a lattice path
from $(0,0)$ to $(n,0)$ never going below the $x$-axis, with these
colored up and down steps and (uncolored) level steps. Such a path
is called a $d$-\emph{Motzkin path} of length $n$.  When $d=1$, it
is an ordinary Motzkin path. When $d\ge 2$, it is equivalent to a
lattice path in $\RR^{d+1}$ mentioned in Theorem \ref{thm:1-1}.
(Note that when $d=2$, the $d$-Motzkin paths defined above is
different from the notion of 2-Motzkin paths appeared in
\cite{Deutsch-Shapiro}, with two kinds of level steps, say
straight and wavy.) In this paper, we will use $d$-Motzkin words
and $d$-Motzkin paths interchangeably.

\smallskip
Let $\M^{(d)}_n$ denote the set of $d$-Motzkin paths of length
$n$. Given a path $\pi=x_1x_2\cdots x_n\in\M^{(d)}_n$, an ordered
pair $(x_i,x_j)=(U^{(k)},D^{(k)})$ is called a \emph{matching
pair} if $j$ is the least integer such that for the subword
$\omega=x_ix_{i+1}\cdots x_j$ of $\pi$,
$|\omega|_{U^{(k)}}=|\omega|_{D^{(k)}}$. For a matching pair
$(x_a,x_b)$ with $i<a<b<j$, we say that $(x_a,x_b)$ is
\emph{nested} in $(x_i,x_j)$. For a level step $x_c$ with $i<c<j$,
we say that $x_c$ is \emph{enclosed} in the matching pair
$(x_i,x_j)$. Note that a level step $x_c$ of $\pi$ is enclosed in
a matching $(U^{(d)},D^{(d)})$-pair if for the prefix
$\mu=x_1\cdots x_c$, $|\mu|_{U^{(d)}}>|\mu|_{D^{(d)}}$.

From (M1) and (M2), the $d$-Motzkin paths satisfy the following two conditions (N1) and (N2) necessarily.

\begin{enumerate}
\item[(N1)] The subword of $\pi$ consisting of the steps $U^{(k)}$ and $D^{(k)}$ is a parenthesis system, for all $1\le k\le d$.
\item[(N2)] Every matching $(U^{(i+1)},D^{(i+1)})$-pair of $\pi$ is nested in a matching $(U^{(i)},D^{(i)})$-pair, for all $1\le i\le d-1$.
\end{enumerate}

\subsection{Main result}

The main result of this paper is a bijection
$\phi_d$ between $\M^{(d)}_n$ and $\T^{(2d+1)}_n$ with a refinement, which gives a proof of Theorem \ref{thm:1-1} and Conjecture \ref{cor:1-2}.

We partition the set $\M^{(d)}_n$ into three subsets $\M^{(d)}_n=\M^{(d-1)}_n\cup \widehat{\M}^{(d)}_n\cup \overline{\M}^{(d)}_n$, where
\begin{eqnarray*}
\widehat{\M}^{(d)}_n   &=& \{\pi\in\M^{(d)}_n: \mbox{ no level steps of $\pi$ are enclosed in a matching $(U^{(d)},D^{(d)})$-pair} \},  \\
\overline{\M}^{(d)}_n &=& \{\pi\in\M^{(d)}_n: \mbox{ some level steps of $\pi$ are enclosed in a matching $(U^{(d)},D^{(d)})$-pair} \}.
\end{eqnarray*}
Note that the paths in $\widehat{\M}^{(d)}_n$ and
$\overline{\M}^{(d)}_n$ contain $(U^{(d)},D^{(d)})$-pairs. We
assume that $\M^{(0)}_n$ consists of the unique path with $n$
level step.  On the other hand, let
$\widetilde{\T}^{(d)}_n\subseteq\T^{(d)}_n$ be the subset of SYTs
with exactly $d$ rows. Here is our main result.

\begin{thm} \label{thm:main-bijection} For $d\ge 1$, there is a bijection $\phi_d$
between $\M^{(d)}_n$ and $\T^{(2d+1)}_n$. Specifically, the map
$\phi_d$ induces a bijection between $\widehat{\M}^{(d)}_n$ and
$\widetilde{\T}^{(2d)}_n$ and a bijection between
$\overline{\M}^{(d)}_n$ and $\widetilde{\T}^{(2d+1)}_n$.
\end{thm}

We shall construct the bijection $\phi_d$ inductively on $d$. For
the initial stage, a map $\phi_1$ has been established in
\cite{Eu}. Using an approach of step-replacement, we give an
alternative construction of $\phi_1$ in section 2, which could be
generalized to $\phi_d$. We describe briefly our strategy for the
construction of $\phi_d$ ($d\ge 2$) in section 3. A precise
algorithm for the construction of $\phi_2$ is given in section 4,
and a precise algorithm for the construction of $\phi_d$ ($d\ge
3$) is given in section 5. For illustrating the constructions, Figure \ref{fig:Motzkin-word} is for $\phi_1$ (Example
\ref{exa:d=1}) and $\phi_1^{-1}$ (Example
\ref{exa:d=1--Backward}), Figure \ref{fig:2-Motzkin-path} is for
$\phi_2$ (Example \ref{exa:d=2}) and $\phi_2^{-1}$ (Example
\ref{exa:d=2--Backward}), and Figures \ref{fig:critical-crucial}
and \ref{fig:critical-up-step-25} are for $\phi_3$ (Example
\ref{exa:d=3}).

\subsection{Related structures of colored Motzkin words} The notion of $d$-Motzkin words is in connection to a
number of objects.

\begin{itemize}
\item \emph{Lazy walks.}
For $1\le i\le d$, if we associate the steps
$U^{(i)}$ and $D^{(i)}$ to $\epsilon_i$ and $-\epsilon_i$,
respectively, and associate the level step to the zero vector,
then the $d$-Motzkin words of length $n$ correspond to Zeilberger's
lattice walks of length $n$ in the region $\{(z_1,\dots,z_d) :
z_1\ge \cdots\ge z_d\ge 0\}\subseteq\ZZ^{d}$, starting and ending at
the origin, where for each step you may either stand still or move
one unit in any direction.
\item \emph{Nonnesting matchings/Oscillating tableaux.}
Restricted to the steps $\{U^{(i)},D^{(i)}: 1\le i\le d\}$, the
notion of $d$-Motzkin paths without level steps is equivalent to the
notion of $(d+1)$-nonnesting matchings and the notion of oscillating
tableaux of empty shape and height $d$, where $U^{(i)}$ is adding
one cell to (respectively, $D^{(i)}$ is removing it from) the $i$th
row. We refer the readers to \cite{Chen-Deng} for their definitions.
\item \emph{Walks in Weyl chamber.}
The paths with strict version of the condition (M2),
$|\mu|_{U^{(1)}}-|\mu|_{D^{(1)}}> \cdots>
|\mu|_{U^{(d)}}-|\mu|_{D^{(d)}}> 0$, are related to the oscillating
lattice walks in the $d$-dimensional Weyl chamber
$\{(z_1,\dots,z_d) : z_1> \cdots> z_d> 0\}\subseteq\ZZ^{d}$ studied
in \cite{Xin}.
\end{itemize}

\section{The initial stage}
We identify a tableau $T\in\T^{(d)}_n$ with its {\em Yamanouchi word}, a word of length $n$ on the alphabet $\{1, 2,\dots,d\}$ the $j$th letter of which is the row index of the cell of $T$ containing the number $j$, for $1\le j\le n$. For example,
\begin{equation*}
T=\young(136,257,48)
\quad \longleftrightarrow \quad 12132123.
\end{equation*}
Note that the $i$th row of $T$ contains the indices of the letters $i$ in its Yamanouchi word.

In this section, we shall demonstrate a construction of the map $\phi_1$ using an approach of step-replacement.

\subsection{The map $\phi_1$}
Given a Motzkin path $\pi=x_1x_2\cdots x_n\in\M^{(1)}_n$, we shall construct a word $\phi_1(\pi)$ of length $n$ by the following procedures, replacing each step $x_i$ by a letter in $\{1,2,3\}$.
Note that a level step is enclosed in a matching $(U^{(1)},D^{(1)})$-pair if and only if it is above the $x$-axis. So, if $\pi\in\overline{\M}^{(1)}_n$ then $\pi$ contains some level steps above the $x$-axis, and if $\pi\in\widehat{\M}^{(1)}_n$ then all of the level steps in $\pi$ are on the $x$-axis.
As the construction proceeds, $\pi$ would become a sequence of steps $\{L,U^{(1)},D^{(1)}\}$ and letters $\{1,2,3\}$. By the path $\pi$ we mean the path-part of $\pi$ whose length is the number of its steps.

\noindent
{\bf Algorithm A.}
\begin{enumerate}
\item[(A1)] $\pi\in\overline{\M}^{(1)}_n$. Then find the first level step above the $x$-axis, say $x_a$, and the first $D^{(1)}$ step after $x_a$, say $x_b$. Form a new sequence $\pi'$
from $\pi$ by replacing the pair $(x_a,x_b)$ by  $(D^{(1)},3)$. If $\pi'$ contains a level step above the $x$-axis then go back to (A1) and proceed to process $\pi'$, otherwise go to (A2).
\item[(A2)] $\pi\in\widehat{\M}^{(1)}_n$. Then find the first $U^{(1)}$ step, say $x_a$, and the first $D^{(1)}$ step after $x_a$, say $x_b$. Form a new sequence $\pi'$ from $\pi$ by replacing the pair $(x_a,x_b)$ by $(L,2)$. If $\pi'$ contains a $U^{(1)}$ step then go back to (A2) and proceed to process $\pi'$, otherwise go to (A3).
\item[(A3)] $\pi$ consists of level steps. Then replace every step of $\pi$ by a letter 1 and we are done.
\end{enumerate}

\begin{exa} \label{exa:d=1} {\rm
 Take the Motzkin path $\pi=x_1\cdots x_{12}$ shown in Figure \ref{fig:Motzkin-word}(a), where the colors of the up and down steps are indicated on their top. For convenience, we follow the notations used in algorithm A. Since the level step $x_3$ is above the $x$-axis, we have $\pi\in\overline{\M}^{(1)}_{12}$. By (A1), we locate the steps $(x_a,x_b)=(x_3,x_4)$ (iteratively, $(x_5,x_7)$). Then we replace $(x_a,x_b)$ by  $(D^{(1)},3)$. The resulting path is in $\widehat{\M}^{(1)}_{10}$; see Figure \ref{fig:Motzkin-word}(b).

  By (A2), we locate the steps $(x_a,x_b)=(x_1,x_3)$ (iteratively, $(x_2,x_5)$, $(x_6,x_8)$, $(x_9,x_{11})$ and $(x_{10},x_{12})$). For each one of these pairs,  we replace $(x_a,x_b)$ by $(L,2)$. We obtain the path shown in Figure \ref{fig:Motzkin-word}(c). Replacing each one of the remaining level steps by a letter 1, we obtain the requested word $\phi_1(\pi)=1 1 2 3 2 1 3 2 1 1 2 2 \in\T^{(3)}_{12}$.
 }
\end{exa}

\begin{figure}[ht]
\begin{center}
\includegraphics[width=2.35in]{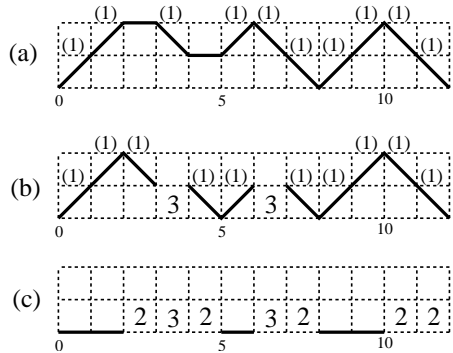}
\end{center}
\caption{\small An illustration for Example \ref{exa:d=1}.} \label{fig:Motzkin-word}
\end{figure}

It is easy to see that the path-part of the new sequence $\pi'$ formed in (A1) and (A2) remains to be a Motzkin path. Hence over iterations, we construct the word $\phi_1(\pi)\in\T^{(3)}_n$. Specifically, $\phi_1(\pi)\in\widetilde{\T}^{(3)}_n$ if $\pi\in\overline{\M}^{(1)}_n$, and $\phi_1(\pi)\in\widetilde{\T}^{(2)}_n$ if $\pi\in\widehat{\M}^{(1)}_n$. For a word $\phi_1(\pi)$ with more than one letter $3$  (respectively, $2$), we observe that the letters 3 (respectively, 2) appear in the order from left to right iteratively.

\subsection{The map $\phi_1^{-1}$}

Given a word $\omega=x_1x_2\cdots x_n\in\T^{(3)}_n$, we recover the Motzkin path $\phi_1^{-1}(\omega)\in\M^{(1)}_n$ by the following procedures.

\noindent
{\bf Algorithm B.}
\begin{enumerate}
\item[(B1)] First, we replace every letter 1 of $\omega$ by a level step.
\item[(B2)] Traversing $\omega$ from right to left, find the rightmost letter 2, say $x_b$, and the nearest level step to the left of $x_b$, say $x_a$. Then replace the pair $(x_a,x_b)$ by $(U^{(1)},D^{(1)})$. Go back to (B2) and repeat the process until there are no letters 2 in $\omega$.
\item[(B3)] Traversing $\omega$ from right to left, find the rightmost letter 3, say $x_b$, and the nearest $D^{(1)}$ step to the left of $x_b$, say $x_a$. Then replace the pair $(x_a,x_b)$ by $(L,D^{(1)})$. Go back to (B3) and repeat the process until there are no letters 3 in $\omega$.
\end{enumerate}
Note that the construction of $\phi_1^{-1}$ is exactly the reverse operation of the construction of $\phi_1$.

\begin{exa} \label{exa:d=1--Backward} {\rm
 Take the word $\omega=x_1x_2\cdots x_{12}=1 1 2 3 2 1 3 2 1 1 2 2 \in\T^{(3)}_{12}$. Replacing the letters 1 by level steps, we obtain the path shown in Figure \ref{fig:Motzkin-word}(c). By (B2), from right to left we replace the pair $(x_a,x_b)=(x_{10},x_{12})$ (iteratively, $(x_{9},x_{11})$, $(x_6,x_8)$, $(x_2,x_5)$ and $(x_1,x_3)$) by $(U^{(1)},D^{(1)})$. The resulting path is shown in Figure \ref{fig:Motzkin-word}(b). By (B3), from right to left we replace the pair $(x_a,x_b)=(x_5,x_7)$ (iteratively,  $(x_3,x_4)$) by $(L,D^{(1)})$. Then we obtain the path $\phi_1^{-1}(\omega)\in\M^{(1)}_{12}$ shown in Figure \ref{fig:Motzkin-word}(a).
 }
\end{exa}

We have established the bijection $\phi_1$ of Theorem \ref{thm:main-bijection}.

\section{Critical up steps and exceeding up steps}
In this section, we describe briefly our strategy for the construction of the map $\phi_d$, using the same approach as in algorithm A.

\subsection{Critical up steps}
For a $d$-Motzkin path $\pi=x_1x_2\cdots x_n$ and $i\le n$, let $\mu(i)$ denote the prefix $x_1\cdots x_i$ of $\pi$. For $2\le k\le d$, an up step $x_j=U^{(k)}$ of $\pi$ is called \emph{critical} if for the prefix $\mu(j)$ of $\pi$,
 \[
 |\mu(j)|_{U^{(k)}}-|\mu(j)|_{D^{(k)}} =|\mu(j)|_{U^{(k-1)}}-|\mu(j)|_{D^{(k-1)}}.
 \]
The critical up steps are important for the construction of $\phi_d$.

Given a path $\pi=x_1x_2\cdots x_n\in\overline{\M}^{(d)}_n$ (respectively, $\pi\in\widehat{\M}^{(d)}_n$),
we start from the first level step enclosed in a matching $(U^{(d)},D^{(d)})$-pair, say $x_a$, (respectively, start from the first $U^{(d)}$) and determine a down-step sequence $(x_{b_d}, x_{b_{d-1}},\dots,x_{b_1})=(D^{(d)},D^{(d-1)},\dots,D^{(1)})$ from left to right such that $x_{b_d}$ is the first $D^{(d)}$ after $x_a$ and $x_{b_{d-1}},\dots,x_{b_1}$ are located by certain greedy procedures (described below). Then we replace $x_a$ by $D^{(d)}$ (respectively, by $L$) and replace the sequence $(x_{b_d}, \dots, x_{b_1})$ by $(D^{(d-1)},\dots,D^{(1)},2d+1)$ (respectively, $(D^{(d-1)},\dots,D^{(1)},2d)$).

To determine $x_{b_{d-1}}$, we search for the nearest $D^{(d-1)}$ after $x_{b_{d}}$.
However, when the pair $(x_a,x_{b_d})=(L,D^{(d)})$ of $\pi$ is replaced by $(D^{(d)},D^{(d-1)})$, the critical $U^{(d)}$ steps $x_j$ with $j>b_d$ would become `illegal' since $|\mu(j)|_{U^{(d)}}-|\mu(j)|_{D^{(d)}} =|\mu(j)|_{U^{(d-1)}}-|\mu(j)|_{D^{(d-1)}}+1$. (This also happens in the situation when $(x_a,x_{b_d})=(U^{(d)},D^{(d)})$ is replaced by $(L,D^{(d-1)})$ in the case $\pi\in\widehat{\M}^{(d)}_n$.) So if we encounter a critical $U^{(d)}$ on the way searching for the nearest $D^{(d-1)}$ after $x_{b_d}$, the critical $U^{(d)}$ and the first $D^{(d)}$ afterwards would be degraded to $(U^{(d-1)},D^{(d-1)})$, respecting the conditions (M1) and (M2) for $d$-Motzkin words. This process goes on until we locate a $D^{(d-1)}$ step, which is our $x_{b_{d-1}}$.
If $d=2$ then the requested down-step sequence associated to $x_a$ is determined; otherwise we move on to determine $x_{b_{d-2}}$, which gets more intricate.

To determine $x_{b_{d-2}}$, we search for the nearest $D^{(d-2)}$ after $x_{b_{d-1}}$. However, when $x_{b_{d-1}}$ is degraded to $D^{(d-2)}$, the critical $U^{(d-1)}$ steps $x_j$ with $j>b_{d-1}$ would become `illegal'. Moreover, once a critical $U^{(d-1)}$ is degraded to $U^{(d-2)}$, the critical $U^{(d)}$ steps afterwards would become `illegal'. We turn the process of determining $x_{b_k}$ into a loop that searches for the nearest $D^{(k)}$ after $x_{b_{k+1}}$. If we encounter a critical $U^{(k+1)}$ on the way, we mark this $U^{(k+1)}$ and go in a loop that searches for the nearest $D^{(k+1)}$ afterwards.

For example, for the 3-Motzkin path $\pi=x_1x_2\cdots x_{25}$ shown in Figure \ref{fig:critical-crucial}(a), the steps $x_4, x_{10}, x_{13}, x_{20}$ are the critical $U^{(2)}$ and  the steps $x_{9}, x_{15}$ are the critical $U^{(3)}$.
We start from $x_a=x_6$ and determine a down-step sequence $(x_{b_3}, x_{b_{2}},x_{b_1})$, which will be $(x_{8},x_{12},x_{23})$ in our approach.

The step $x_{b_3}=x_8$ is the first  $D^{(d)}$ after $x_a$. To determine $x_{b_2}$, we search for the nearest $D^{(2)}$ after $x_{b_3}$ and encounter a critical $U^{(3)}$ step, $x_9$, on the way. We mark $x_9$ and the first $D^{(3)}$ step, $x_{11}$, afterwards, then we get to a $D^{(2)}$ step, $x_{12}$, which is $x_{b_2}$. By the end of this iteration, the pair $(x_{9},x_{11})$ will be degraded to $(U^{(2)},D^{(2)})$.

To determine $x_{b_1}$, we search for the nearest $D^{(1)}$ after $x_{b_2}$ and encounter a critical $U^{(2)}$ step, $x_{13}$. Then mark this $U^{(2)}$ and search for the nearest $D^{(2)}$ as above. On the way, we encounter a critical $U^{(3)}$ step, $x_{15}$. The pairs $(x_{15},x_{17})$  and $(x_{13},x_{19})$ will be degraded to $(U^{(2)},D^{(2)})$ and $(U^{(1)},D^{(1)})$, respectively. Then before getting to the step $x_{23}=D^{(1)}$, which is $x_{b_1}$, we encounter another critical $U^{(2)}$ step, $x_{20}$, and the pair $(x_{20},x_{21})$ will be degraded to $(U^{(1)},D^{(1)})$.  The updated path is shown in Figure \ref{fig:critical-crucial}(b).

\begin{figure}[ht]
\begin{center}
\includegraphics[width=4.75in]{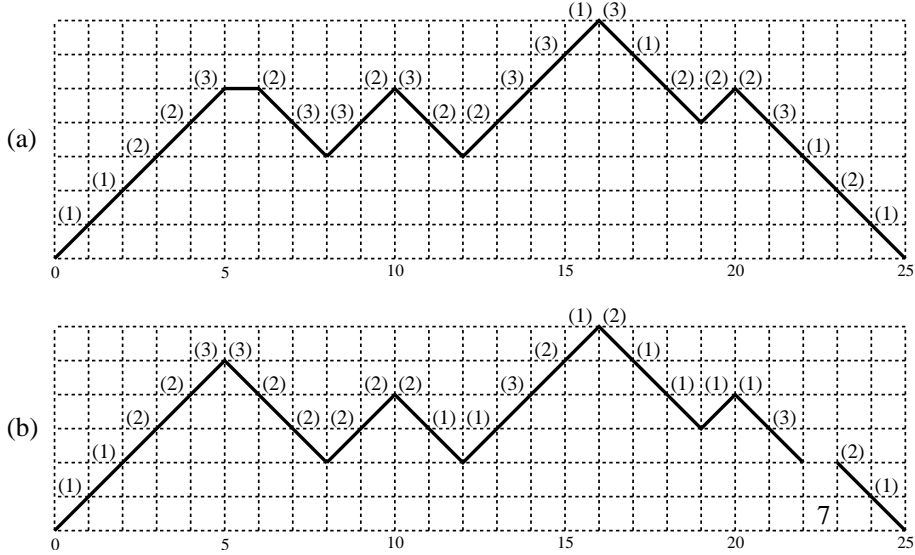}
\end{center}
\caption{\small A 3-Motzkin path in $\overline{\M}^{(3)}_{25}$ and a 3-Motzkin path in  $\widehat{\M}^{(3)}_{24}$.} \label{fig:critical-crucial}
\end{figure}



\subsection{Exceeding up steps}
For a $d$-Motzkin path $\pi=x_1x_2\cdots x_n$ and for $1\le k\le d-1$, an up step $x_j=U^{(k)}$ of $\pi$ is called \emph{exceeding} if for the prefix $\mu(j)$ of $\pi$,
\[
|\mu(j)|_{U^{(k)}}-|\mu(j)|_{D^{(k)}} =|\mu(j)|_{U^{(k+1)}}-|\mu(j)|_{D^{(k+1)}}+1.
\]
For example, for the path $\pi'=x_1x_2\cdots x_{25}$ (where $x_{23}=7$) shown in Figure \ref{fig:critical-crucial}(b), the steps $x_1,x_{13}, x_{16}, x_{20}$ are the exceeding $U^{(1)}$ and the steps $x_3, x_9, x_{15}$ are the exceeding $U^{(2)}$.
As one has noticed, the updated critical $U^{(k)}$ steps of the path $\pi$ in Figure \ref{fig:critical-crucial}(a) become exceeding $U^{(k-1)}$ steps of the path $\pi'$ in Figure \ref{fig:critical-crucial}(b).

The exceeding up steps are crucial for the construction of $\phi_d^{-1}$. The construction is devised as the reverse operation of the construction of $\phi_d$. Let us see how the exceeding up steps are used to recover the path $\pi$ in Figure \ref{fig:critical-crucial}(a) from the path $\pi'$ in Figure \ref{fig:critical-crucial}(b).

Starting from the letter $x_{23}=7$ and the nearest $D^{(1)}$ step, $x_{21}$, to the left, we traverse from right to left and search the nearest $D^{(2)}$. If we encounter an exceeding $U^{(1)}$ step on the way, we mark this $U^{(1)}$ and the nearest $D^{(1)}$ to the left. This process goes on until a $D^{(2)}$ is found. In Figure \ref{fig:critical-crucial}(b), we mark the pair $(x_{19},x_{20})$ and get to $x_{17}=D^{(2)}$.

Then search for the nearest $D^{(3)}$ to the left. If we encounter an exceeding $U^{(2)}$ step on the way, we mark this $U^{(2)}$ and go in a loop that search for the nearest $D^{(2)}$ as above. In Figure \ref{fig:critical-crucial}(b), to the left of $x_{17}$ we encounter an exceeding $U^{(2)}$ step, $x_{15}$. Then mark this $U^{(2)}$ and search for the nearest $D^{(2)}$ to the left. On the way, we encounter an exceeding $U^{(1)}$ step, $x_{13}$. We mark the pair $(x_{12},x_{13})$ and then locate a $D^{(2)}$ step, $x_{11}$. Before we get to the $D^{(3)}$ step $x_a=x_6$, we encounter another exceeding $U^{(2)}$ step, $x_{9}$. This time we locate a $D^{(2)}$ step, $x_8$. Hence all of the updated steps of $\pi$ are retrieved.

\section{The bijection $\phi_2$}
In this section, we give precise constructions of $\phi_2$ and $\phi_2^{-1}$.

\subsection{The map $\phi_2$} Let $\pi=x_1x_2\cdots x_n\in\M^{(2)}_n$. In the following algorithm, (C1) and (C2) are devised as an iteration for the construction of $\phi_2$.
  At the beginning of each iteration, the critical up steps of $\pi$ are refreshed. We start from the first level step, say $x_a$, enclosed in a matching $(U^{(2)},D^{(2)})$-pair (respectively, the first $U^{(2)}$) and traverse from left to right to determine down steps $(x_{b_2},x_{b_1})=(D^{(2)},D^{(1)})$.
  Meanwhile, in the subword between $x_{b_2}$ and $x_{b_1}$, we mark the critical $U^{(2)}$ steps and their nearest $D^{(2)}$ steps in a greedy way.

Likewise, over iterations $\pi$ would become a sequence of steps $\{L,U^{(1)},U^{(2)},D^{(1)},D^{(2)}\}$ and letters $\{1,2,\dots,5\}$. By the path $\pi$ we mean the path-part of $\pi$ whose length is the number of its steps.

\noindent
{\bf Algorithm C.}
 \begin{enumerate}
 \item[(C1)] $\pi\in\overline{\M}^{(2)}_n$. Then find the first level step enclosed in a matching $(U^{(2)},D^{(2)})$-pair, say $x_a$, and find the first $D^{(2)}$ after $x_a$, say $x_{b_2}$. Searching for the nearest $D^{(1)}$ after $x_{b_2}$, if we encounter a critical $U^{(2)}$ before a $D^{(1)}$ then we mark this $U^{(2)}$ and the first $D^{(2)}$ afterwards, say
 $x_{q}$. Searching for the nearest $D^{(1)}$ after $x_{q}$, if we encounter a critical $U^{(2)}$ before a $D^{(1)}$ then we also mark this $U^{(2)}$ and the first $D^{(2)}$ afterwards. Repeat this process until a $D^{(1)}$ is found, say $x_{b_1}$.
    Then form a new sequence $\pi'$ from $\pi$ as follows. Replace $(x_a,x_{b_2},x_{b_1})$ by $(D^{(2)},D^{(1)},5)$ and degrade
 every marked $(U^{(2)},D^{(2)})$-pair between $x_{b_2}$ and $x_{b_1}$ to $(U^{(1)},D^{(1)})$. If $\pi'$ contains a level step enclosed in a matching $(U^{(2)},D^{(2)})$-pair then go back to (C1) and proceed to process $\pi'$, otherwise go to (C2).
  \item[(C2)] $\pi\in\widehat{\M}^{(2)}_n$. Then find the first $U^{(2)}$, say $x_a$, and find the first $D^{(2)}$ after $x_a$, say $x_{b_2}$. Do the same procedure as in (C1) to search for the nearest $D^{(1)}$ after $x_{b_2}$. We end up with the located $D^{(1)}$, say $x_{b_1}$,
 and a number of marked $(U^{(2)},D^{(2)})$-pairs (possibly empty) between $x_{b_2}$ and $x_{b_1}$. Then we form a new sequence $\pi'$ from $\pi$ as follows. Replace $(x_a,x_{b_2},x_{b_1})$ by $(L,D^{(1)},4)$ and degrade every marked $(U^{(2)},D^{(2)})$-pair between $x_{b_2}$ and $x_{b_1}$ by $(U^{(1)}, D^{(1)})$. If $\pi'$ contains a $U^{(2)}$ step then go back to (C2) and proceed to process $\pi'$, otherwise go to (C3).
 \item[(C3)] $\pi\in\M^{(1)}_n$. Then the word $\phi_2(\pi)=\phi_{1}(\pi)$ can be obtained by algorithm A.
 \end{enumerate}

\begin{exa} \label{exa:d=2} {\rm
 Take the 2-Motzkin path $\pi=x_1x_2\cdots x_{15}$ shown in Figure \ref{fig:2-Motzkin-path}(a).
 We follow the notations used in algorithm C. By (C1), we locate the steps $(x_a,x_{b_2})=(x_7,x_8)$. Searching for the nearest $D^{(1)}$ after $x_{b_2}$, we encounter a critical $U^{(2)}$ step, $x_{11}$. So mark $x_{11}$ and the first $D^{(2)}$ step, $x_{12}$, afterwards. Then we locate a $D^{(1)}$ step, $x_{13}$, which is our $x_{b_1}$. Replace $(x_a,x_{b_2},x_{b_1})$ by $(D^{(2)},D^{(1)},5)$ and degrade the marked pair $(x_{11},x_{12})$ to  $(U^{(1)},D^{(1)})$. The resulting path $\pi'\in\widehat{\M}^{(2)}_{14}$ is shown in Figure \ref{fig:2-Motzkin-path}(b).

 By (C2), we locate the steps $(x_a,x_{b_2})=(x_3,x_4)$. Searching for the nearest $D^{(1)}$ after $x_{b_2}$, we encounter a critical $U^{(2)}$ step, $x_6$. So mark $x_6$ and the first $D^{(2)}$ step, $x_7$, afterwards. Then we locate a $D^{(1)}$ step, $x_8$, which is our $x_{b_1}$.
 Replace $(x_a,x_{b_2},x_{b_1})$ by  $(L,D^{(1)},4)$ and degrade the marked pair $(x_6,x_7)$ to  $(U^{(1)},D^{(1)})$. The resulting path $\pi''\in\widehat{\M}^{(2)}_{13}$ is shown in Figure \ref{fig:2-Motzkin-path}(c). By (C2) once again, we replace $(x_a,x_{b_2},x_{b_1})=(x_5,x_9,x_{15})$ by $(L,D^{(1)},4)$ and degrade the pair $(x_{10},x_{14})$ to $(U^{(1)},D^{(1)})$. The resulting Motzkin path is shown in Figure \ref{fig:2-Motzkin-path}(d). The letters of the remaining steps can be determined by (A1)-(A3), as shown in Example \ref{exa:d=1}. We obtain the corresponding word $\phi_2(\pi)=1 1 2 3 2 1 3 4 2 1 1 2 5 2 4\in\T^{(5)}_{15}$.
 }
\end{exa}

\begin{figure}[ht]
\begin{center}
\includegraphics[width=2.85in]{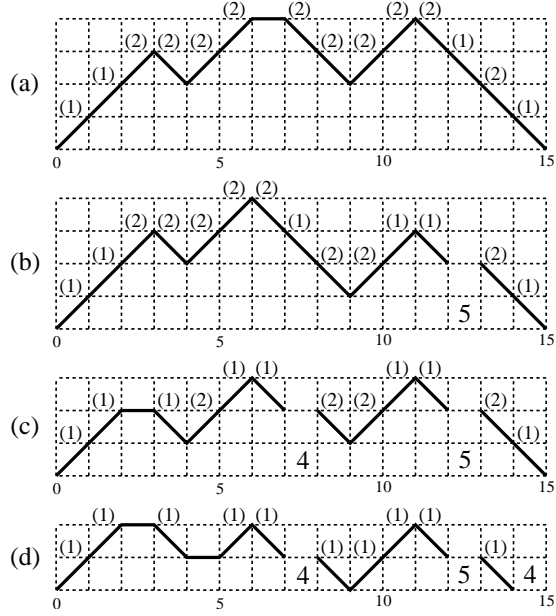}
\end{center}
\caption{\small A 2-Motzkin path $\pi$ and the corresponding word $\phi_2(\pi)$.} \label{fig:2-Motzkin-path}
\end{figure}

We shall prove that the word $\phi_2(\pi)$ constructed above is a Yamanouchi word in $\T^{(5)}_n$.

\begin{pro} \label{pro:one-iteration}
For a path $\pi\in\M^{(2)}_n$, we have $\phi_2(\pi)\in\T^{(5)}_n$. Specifically, $\phi_2(\pi)\in\widetilde{\T}^{(5)}_n$ if $\pi\in\overline{\M}^{(2)}_n$,  and $\phi_2(\pi)\in\widetilde{\T}^{(4)}_n$ if $\pi\in\widehat{\M}^{(2)}_n$.
\end{pro}

 \begin{proof}
 Let $\pi=x_1\cdots x_n\in\overline{\M}^{(2)}_n$, and let $\pi'$ be the sequence obtained from $\pi$  by (C1) in one iteration. First, we shall prove the following property.

 CLAIM:  $\pi'$ is a 2-Motzkin path, i.e., $\pi'\in\M^{(2)}_{n-1}$.

For the path $\pi$, let $(x_a,x_{b_2},x_{b_1})$ be the triplet $(L,D^{(2)},D^{(1)})$ located in (C1), and let $(x_{p_1},x_{q_1}),\dots,(x_{p_m},x_{q_m})$ be the marked $(U^{(2)},D^{(2)})$-pairs from left to right between $x_{b_2}$ and $x_{b_1}$,  for some $m\ge 0$. Note that $x_{p_i}$ is the first critical $U^{(2)}$ after $x_{q_{i-1}}$, for $1\le i\le m$. (We assume $(p_0,q_0)=(a,b_2)$.)

 The path $\pi'$ is obtained from $\pi$ by replacing $(x_a,x_{b_2},x_{b_1})$ by $(D^{(2)},D^{(1)},5)$ and degrading each $(x_{p_i},x_{q_i})$ to $(U^{(1)},D^{(1)})$, $1\le i\le m$. We shall verify that the path $\pi'$ satisfies the conditions (M1) and (M2) for 2-Motzkin paths.

 It is easy to see that $\pi'$  satisfies (M1). Suppose that $\pi'$ is against (M2).
 Let $j$ be the least integer such that for the prefix $\mu'=x_1\cdots x_j$ of $\pi'$,
\begin{equation} \label{eqn:against}
 |\mu'|_{U^{(2)}}-|\mu'|_{D^{(2)}}=|\mu'|_{U^{(1)}}-|\mu'|_{D^{(1)}}+1.
\end{equation}
 Then $x_j$ is either a $U^{(2)}$ or a $D^{(1)}$. Let $\mu$ be the prefix of $\pi$ with $|\mu|=|\mu'|$. Consider the following three cases.

 Case I. $q_{i-1} <j< p_i$ for some $i$. If $x_j$ is a $D^{(1)}$ then $x_j$ would be the $D^{(1)}$ we located, a contradiction. Otherwise, $x_j$ is a
 $U^{(2)}$.  Comparing $\mu$ and $\mu'$ (see Figure \ref{fig:proof-4-2}), we observe that
 $|\mu|_{U^{(2)}}-|\mu|_{D^{(2)}}=|\mu'|_{U^{(2)}}-|\mu'|_{D^{(2)}}$ and $|\mu|_{U^{(1)}}-|\mu|_{D^{(1)}}=|\mu'|_{U^{(1)}}-|\mu'|_{D^{(1)}}+1$. From Eq.\,(\ref{eqn:against}), it follows that
 $|\mu|_{U^{(2)}}-|\mu|_{D^{(2)}}=|\mu|_{U^{(1)}}-|\mu|_{D^{(1)}}$, and hence $x_j$ is critical. This contradicts that $x_{p_i}$ is the first critical $U^{(2)}$  after $x_{q_{i-1}}$.

 Case II. $p_i <j< q_i$ for some $i$. Comparing $\mu$ and $\mu'$, we obtain that
 $|\mu|_{U^{(2)}}-|\mu|_{D^{(2)}}=|\mu|_{U^{(1)}}-|\mu|_{D^{(1)}}+2$. This contradicts that $\pi$ is a 2-Motzkin path.

 Case III. $b_1<j$. Comparing $\mu$ and $\mu'$, we obtain that
 $|\mu|_{U^{(2)}}-|\mu|_{D^{(2)}}=|\mu|_{U^{(1)}}-|\mu|_{D^{(1)}}+1$. This also contradicts that $\pi$ is a 2-Motzkin path. The claim is proved.

 Over iterations, we construct the word $\phi_2(\pi)$.
 At a certain iteration, the step $x_a=D^{(2)}$ of $\pi'$  would be applied in (C2), getting a $D^{(1)}$ step, say $x_k$ with $a<k\le b_2$, replaced by a letter 4 (note that there are no critical $U^{(2)}$ steps between $x_a$ and $x_{b_2}$). Moreover, whenever a step gets replaced by a letter 5 there is always an accompanying pair $(x_{a'},x_{b'_2})$ getting a $D^{(1)}$ step $x_{k'}$ with $a'<k'\le b'_2$ replaced by a letter 4. Hence $\phi_2(\pi)\in\widetilde{\T}^{(5)}_n$.

 By a similar argument, one can prove the other case $\pi\in\widehat{\M}^{(2)}_n$. The assertion follows.
\end{proof}

\begin{figure}[ht]
\begin{center}
\psfrag{A}[][][1]{$\pi:$}
\psfrag{B}[][][1]{$\pi':$}
\psfrag{a}[][][1]{$a$}
\psfrag{b2}[][][1]{$b_2$}
\psfrag{b1}[][][1]{$b_1$}
\psfrag{pa}[][][1]{$p_{i-1}$}
\psfrag{j}[][][1]{$j$}
\psfrag{qa}[][][1]{$q_{i-1}$}
\psfrag{pi}[][][1]{$p_{i}$}
\psfrag{qi}[][][1]{$q_{i}$}
\includegraphics[width=4in]{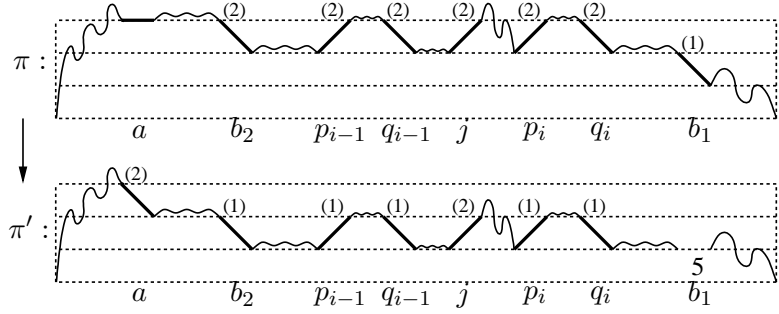}
\end{center}
\caption{\small Illustration for Case I in the proof of Proposition \ref{pro:one-iteration}.} \label{fig:proof-4-2}
\end{figure}

For a path $\pi\in\M^{(2)}_n$ with more than one step replaced by a letter 5 (respectively, 4), the following result shows that these steps got replaced in the order from left to right.

\begin{pro} \label{pro:Left->Right} For $j\ge 2$ and a word $\phi_2(\pi)$ with $j$ letters of 5 (respectively, 4), let $x_{i_1},\dots,x_{i_j}$ be the steps of $\pi$ replaced by a letter 5 (respectively, 4) in iterative order. Then $i_1<\cdots<i_j$.
\end{pro}

\begin{proof} With abuse of notation, let $\pi$ and $\pi'$ be the paths in two consecutive iterations of (C1). For $\pi$, let $(x_a,x_{b_2},x_{b_1})$ be the triplet $(L,D^{(2)},D^{(1)})$ located in (C1), and let $(x_{p_1},x_{q_1}),\dots,(x_{p_m},x_{q_m})$ be the marked $(U^{(2)},D^{(2)})$-pairs between $x_{b_2}$ and $x_{b_1}$,  for some $m\ge 0$. We assume $(p_0,q_0)=(a,b_2)$.
The path $\pi'$ is obtained from $\pi$ by replacing $(x_a,x_{b_2},x_{b_1})$ by $(D^{(2)},D^{(1)},5)$ and degrading each $(x_{p_i},x_{q_i})$ to $(U^{(1)},D^{(1)})$, $1\le i\le m$.

For $\pi'$, let $(x_{a'},x_{b'_2},x_{b'_1})$ be the triplet $(L,D^{(2)},D^{(1)})$ located in (C1), and let  $(x_{s_1},x_{t_1}),\dots$, $(x_{s_{r}},x_{t_{r}})$ be the marked $(U^{(2)},D^{(2)})$-pairs between $x_{b'_2}$ and $x_{b'_1}$, for some $r\ge 0$. We shall prove that $b_1<b'_1$.

Since $x_{b_1}$ is the nearest $D^{(1)}$ after $x_{q_m}$ in the former iteration, it suffices to prove the following property.

CLAIM:  $x_{b'_1}$ is after $x_{q_m}$, i.e., $q_m<b'_1$.

Note that $b_2<b'_2$ (since they are determined by the same
operation as in (A1) of the map $\phi_1$), so if $m=0$ or
$q_m<b'_2$ ($m\ge 1$) then the claim is proved. Let $b'_2<q_m$.
Since $x_{q_i}$ is the first $D^{(2)}$ after $x_{p_i}$ in $\pi$
($1\le i\le m$), we have $q_{h-1}<b'_2<p_h$ for some $h\ge 1$.
Then there exists a critical $U^{(2)}$ between $x_{b'_2}$ and
$x_{p_h}$ in $\pi'$. (Note that there are no $D^{(1)}$ steps
between $x_{q_{h-1}}$ and $x_{p_{h}}$ in $\pi$, or otherwise we
would locate a $D^{(1)}$ before $x_{p_h}$ in the former
iteration.)

Suppose to the contrary that $b'_1<q_m$. Then for the last marked pair $(x_{s_r},x_{t_r})$ of $\pi'$, we have $q_{k-1}<t_r<p_{k}$ for some $k\le m$.
Consider the relative order of $x_{p_{k-1}}, x_{q_{k-1}}, x_{s_r}$ and $x_{t_r}$. Since $x_{s_r}$ is a $U^{(2)}$ in $\pi'$, it did not get replaced in the former iteration. There are two cases. For some $j$ ($h\le j\le k$),

(i) $p_{j-1}<s_r<q_{j-1}$ and $q_{k-1}<t_r<p_{k}$.  By (M2), we
have $|\mu(s_r)|_{U^{(2)}}-|\mu(s_r)|_{D^{(2)}}\le
|\mu(s_r)|_{U^{(1)}}-|\mu(s_r)|_{D^{(1)}}$.

(ii) $q_{j-1}<s_r<p_{j}$ and $q_{k-1}<t_r<p_{k}$ (see Figure \ref{fig:proof-4-3}). Then $x_{s_r}$ is not critical in $\pi$ and we have $|\mu(s_r)|_{U^{(2)}}-|\mu(s_r)|_{D^{(2)}}\le |\mu(s_r)|_{U^{(1)}}-|\mu(s_r)|_{D^{(1)}}-1$.

In either case, it follows that  $|\mu(t_r)|_{U^{(2)}}-|\mu(t_r)|_{D^{(2)}}\le |\mu(t_r)|_{U^{(1)}}-|\mu(t_r)|_{D^{(1)}}-2$.
Since $x_{p_{k}}$ is a critical $U^{(2)}$ in $\pi$, $|\mu(p_k)|_{U^{(2)}}-|\mu(p_k)|_{D^{(2)}}= |\mu(p_k)|_{U^{(1)}}-|\mu(p_k)|_{D^{(1)}}$. Then there exists a $U^{(2)}$, say $x_{\ell}$, between $x_{t_r}$ and $x_{p_{k}}$ such that
\begin{equation} \label{eqn:ell}
|\mu(\ell)|_{U^{(2)}}-|\mu(\ell)|_{D^{(2)}}=|\mu(\ell)|_{U^{(1)}}-|\mu(\ell)|_{D^{(1)}}-1
\end{equation}
(since there are no $D^{(1)}$ steps between $x_{q_{k-1}}$ and $x_{p_{k}}$ in $\pi$).
Comparing $\mu(\ell)$ and $\mu'(\ell)$, we observe that
 $|\mu'(\ell)|_{U^{(2)}}-|\mu'(\ell)|_{D^{(2)}}=|\mu(\ell)|_{U^{(2)}}-|\mu(\ell)|_{D^{(2)}}$ and $|\mu'(\ell)|_{U^{(1)}}-|\mu'(\ell)|_{D^{(1)}}=|\mu(\ell)|_{U^{(1)}}-|\mu(\ell)|_{D^{(1)}}-1$.
From Eq.\,(\ref{eqn:ell}), it follows that
 $|\mu'(\ell)|_{U^{(2)}}-|\mu'(\ell)|_{D^{(2)}}= |\mu'(\ell)|_{U^{(1)}}-|\mu'(\ell)|_{D^{(1)}}$. Such a step $x_\ell$ is a critical $U^{(2)}$ in $\pi'$. It follows that $b'_1<\ell$ since $(x_{s_r},x_{t_r})$ is the last marked pair before $x_{b'_1}$ in $\pi'$. This implies that $x_{b'_1}$ is a $D^{(1)}$ between $x_{q_{k-1}}$ and $x_{p_{k}}$, and that we would get $x_{b'_1}$ replaced by a letter 5 in the former iteration, a contradiction.
The claim is proved.

This proves the assertion for the case $\phi_2(\pi)\in\widetilde{\T}^{(5)}_n$.
By a similar argument, one can prove the assertion for the other case $\phi_2(\pi)\in\widetilde{\T}^{(4)}_n$.
\end{proof}

\begin{figure}[ht]
\begin{center}
\psfrag{A}[][][1]{$\pi:$}
\psfrag{B}[][][1]{$\pi':$}
\psfrag{a}[][][1]{$a$}
\psfrag{b2}[][][1]{$b_2$}
\psfrag{b1}[][][1]{$b_1$}
\psfrag{pb}[][][1]{$p_{j-1}$}
\psfrag{ell}[][][1]{$\ell$}
\psfrag{qb}[][][1]{$q_{j-1}$}
\psfrag{pi}[][][1]{$p_{i}$}
\psfrag{qi}[][][1]{$q_{i}$}
\psfrag{pk}[][][1]{$p_{k}$}
\psfrag{qk}[][][1]{$q_{k}$}
\psfrag{sr}[][][1]{$s_{r}$}
\psfrag{tr}[][][1]{$t_{r}$}
\includegraphics[width=5.25in]{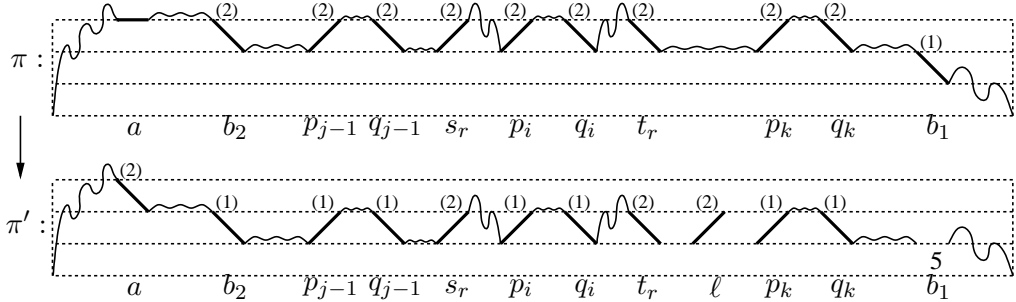}
\end{center}
\caption{\small Illustration for case (ii) in the proof of Proposition \ref{pro:Left->Right}, where $j\le i\le k-1$ (possibly empty).} \label{fig:proof-4-3}
\end{figure}

\subsection{The map $\phi_2^{-1}$}

Given a word $\omega=x_1x_2\cdots x_n\in\T_n^{(5)}$, the Motzkin path corresponding to the subword of $\omega$ consisting of the letters $\{1,2,3\}$ can be determined by the map $\phi_1^{-1}$, using algorithm B. For the letters 4 (respectively, 5) of $\omega$, we shall recover their steps in the order from right to left, using the reverse operation of algorithm C.

At the beginning of each iteration,  the exceeding up steps of $\omega$ are refreshed. We start from the rightmost letter 4 (respectively, 5) of $\omega$ and the first $D^{(1)}$ step to the left. Traverse $\omega$ from right to left to locate a level step above the $x$-axis (respectively, a $U^{(2)}$). On the way, we mark the exceeding $U^{(1)}$ steps and their nearest $D^{(1)}$ to the left in a greedy way, managing to recover the steps that have been updated in (C1) and (C2).

\noindent
{\bf Algorithm D.}

\begin{enumerate}
\item[(D1)] If $\omega$ contains no letters 4 then the path $\phi_2^{-1}(\omega)=\phi_{1}^{-1}(\omega)$ is obtained. Otherwise, go to (D2).
\item[(D2)] If $\omega$ contains a letter 4, then find the rightmost letter $4$, say $x_{b}$, and from right to left find the nearest $D^{(1)}$ to the left of $x_{b}$, say $x_q$. Searching for the nearest level step to the left of $x_q$, if we encounter an exceeding $U^{(1)}$ then we mark this $U^{(1)}$ and the nearest $D^{(1)}$ on its left, say $x_{q'}$. Searching for the nearest level step to the left of $x_{q'}$, if we encounter an exceeding $U^{(1)}$ then we also mark this $U^{(1)}$ and the nearest $D^{(1)}$ on its left. Repeat this process until a level step is found, say $x_{a}$.
    Then  form a new sequence $\omega'$ from $\omega$ as follows. Replace $(x_a,x_{b})$ by $(U^{(2)},D^{(1)})$ and upgrade every marked $D^{(1)}$ and $U^{(1)}$ between $x_a$ and $x_{b}$ to $D^{(2)}$ and $U^{(2)}$, respectively. If $\omega'$ contains no letter 4 then go to (D3); otherwise, go back to (D2) and proceed to process $\omega'$.
\item[(D3)] If $\omega$ contains a letter 5 then  find the rightmost letter $5$, say $x_{b}$, and from right to left find the nearest $D^{(1)}$ to the left of $x_{b}$, say $x_q$. Use the same procedure as in (D2) to search for the nearest $D^{(2)}$ to the left of $x_{q}$. We end up with the located $x_a=D^{(2)}$ and a number of marked $D^{(1)}$ and $U^{(1)}$ between $x_a$ and $x_{b}$. Then form a new sequence $\omega'$ from $\omega$ as follows. Replace $(x_a,x_{b})$ by  $(L,D^{(1)})$ and upgrade every marked $D^{(1)}$ and $U^{(1)}$ between $x_a$ and $x_{b}$ to $D^{(2)}$ and $U^{(2)}$, respectively. If $\omega'$ contains no letter 5 then we are done; otherwise, go back to (D3) and proceed to process $\omega'$.
\end{enumerate}

We remark that a $U^{(1)}$ step that rises from the $x$-axis to the line $y=1$ is always exceeding, so the level step $x_a$ located in (D2) is always above the $x$-axis.

\begin{exa} \label{exa:d=2--Backward} {\rm Take the word $\omega=1 1 2 3 2 1 3 4 2 1 1 2 5 2 4\in\T^{(5)}_{15}$. Using algorithm B, we obtain the Motzkin path corresponding to the subword of $\omega$ consisting of the letters $\{1,2,3\}$, as shown in Figure \ref{fig:2-Motzkin-path}(d). By (D2), we start from the rightmost letter 4, i.e., $x_{b}=x_{15}$, and find the nearest $D^{(1)}$ step, $x_{14}$, to the left. Searching for the nearest level step to the left of $x_{14}$, we encounter an exceeding $U^{(1)}$ step, $x_{10}$. So mark $x_{10}$ and the nearest $D^{(1)}$ to the left, i.e., $x_9$. Then we get to the level step $x_a=x_5$.  Replace $(x_a,x_{b})$ by $(U^{(2)},D^{(1)})$ and upgrade  $(x_9,x_{10},x_{14})$ to $(D^{(2)},U^{(2)},D^{(2)})$. The resulting path $\omega'$ is shown in Figure \ref{fig:2-Motzkin-path}(c). By (D2) once again, we obtain a new path from $\omega'$, which is shown in Figure \ref{fig:2-Motzkin-path}(b).

By (D3), we start from the rightmost letter 5, i.e., $x_{b}=x_{13}$, and find the nearest $D^{(1)}$ step, $x_{12}$, to the left. Searching for the nearest $D^{(2)}$ to the left of $x_{12}$, we encounter an exceeding $U^{(1)}$ step, $x_{11}$. So mark $x_{11}$ and the nearest $D^{(1)}$ to the left, i.e., $x_8$. Then we get to the $D^{(2)}$ step $x_a=x_7$.  Replace $(x_a,x_{b})$ by  $(L,D^{(1)})$ and upgrade $(x_8,x_{11},x_{12})$ to $(D^{(2)},U^{(2)},D^{(2)})$. The resulting path $\phi_2^{-1}(\omega)$ is shown in Figure \ref{fig:2-Motzkin-path}(a).
}
\end{exa}

The following proposition shows that the algorithm D is the reverse operation of the algorithm C.

\begin{pro} \label{pro:d=2--Backward} For a path $\pi\in\M^{(2)}_n$, let $\pi'$ be the path obtained from $\pi$  by (C1) (respectively, (C2)) in one iteration. Then $\pi$ can be recovered from $\pi'$  by (D3) (respectively, (D2)).
\end{pro}

\begin{proof} For a path $\pi=x_1\cdots x_n\in\overline{\M}^{(2)}_n$, let $(x_a,x_{b_2},x_{b_1})=(L,D^{(2)},D^{(1)})$ be the triplet of $\pi$ located in (C1), and let $(x_{p_1},x_{q_1}),\dots,(x_{p_m},x_{q_m})$ be the marked $(U^{(2)},D^{(2)})$-pairs between $x_{b_2}$ and $x_{b_1}$, for some $m\ge 0$, where $x_{p_i}$ is the first critical $U^{(2)}$ after $x_{q_{i-1}}$, for $1\le i\le m$. (We assume $(p_0,q_0)=(a,b_2)$.)
 The path $\pi'$ is obtained from $\pi$ by replacing $(x_a,x_{b_2},x_{b_1})$ by $(D^{(2)},D^{(1)},5)$ and degrading each $(x_{p_i},x_{q_i})$ to $(U^{(1)},D^{(1)})$, $1\le i\le m$.

 To recover $\pi$ from $\pi'$, we first restructure the subword from $x_{b_1}$ to $x_{b_2}$.
By (D3), we start from the letter $x_{b_1}$. Traversing $\pi'$ from right to left, we observe that $x_{q_m}$ is the nearest $D^{(1)}$ to the left of $x_{b_1}$ and that each $x_{q_{i-1}}$ is the nearest $D^{(1)}$ to the left of $x_{p_i}$ for $1\le i\le m$, or otherwise we would get to a $D^{(1)}$ before $x_{b_1}$ in $\pi$. It suffices to prove the following property of $\pi'$.

CLAIM: each $x_{p_i}$ is the nearest exceeding $U^{(1)}$ to the left of $x_{q_i}$ in $\pi'$.

Suppose it is not, let $x_j$ be an exceeding $U^{(1)}$ between $x_{p_i}$ and  $x_{q_i}$ for some $i$ (see Figure \ref{fig:proof-4-5}). Then for the prefix $\mu'(j)$ of $\pi'$, $|\mu'(j)|_{U^{(1)}}-|\mu'(j)|_{D^{(1)}} =|\mu'(j)|_{U^{(2)}}-|\mu'(j)|_{D^{(2)}}+1$.  Comparing $\mu'(j)$ and the prefix $\mu(j)$ of $\pi$, we have
$|\mu(j)|_{U^{(2)}}-|\mu(j)|_{D^{(2)}}=|\mu'(j)|_{U^{(2)}}-|\mu'(j)|_{D^{(2)}}+1$ and $|\mu(j)|_{U^{(1)}}-|\mu(j)|_{D^{(1)}}=|\mu'(j)|_{U^{(1)}}-|\mu'(j)|_{D^{(1)}}$. It follows that $|\mu(j)|_{U^{(1)}}-|\mu(j)|_{D^{(1)}}=|\mu(j)|_{U^{(2)}}-|\mu(j)|_{D^{(2)}}$. Since $x_j$ is a $U^{(1)}$, we have $|\mu(j-1)|_{U^{(1)}}-|\mu(j-1)|_{D^{(1)}}=|\mu(j-1)|_{U^{(2)}}-|\mu(j-1)|_{D^{(2)}}-1$. This contradicts that $\pi$ is a 2-Motzkin path. The claim is proved, and the subword from $x_{b_1}$ to $x_{b_2}$ is recovered.

By (C1), there are no other $D^{(2)}$ steps between $x_a$ and $x_{b_2}$ in $\pi$. Hence we can locate $x_a$ by finding the nearest $D^{(2)}$ step to the left of $x_{b_2}$ in $\pi'$.
This proves the assertion for $\pi\in\overline{\M}^{(2)}_n$.
By a similar argument, one can prove the assertion for $\pi\in\widehat{\M}^{(2)}_n$.
\end{proof}

\begin{figure}[ht]
\begin{center}
\psfrag{A}[][][1]{$\pi:$}
\psfrag{B}[][][1]{$\pi':$}
\psfrag{a}[][][1]{$a$}
\psfrag{b2}[][][1]{$b_2$}
\psfrag{b1}[][][1]{$b_1$}
\psfrag{pa}[][][1]{$p_{i-1}$}
\psfrag{j}[][][1]{$j$}
\psfrag{qa}[][][1]{$q_{i-1}$}
\psfrag{pi}[][][1]{$p_{i}$}
\psfrag{qi}[][][1]{$q_{i}$}
\includegraphics[width=4in]{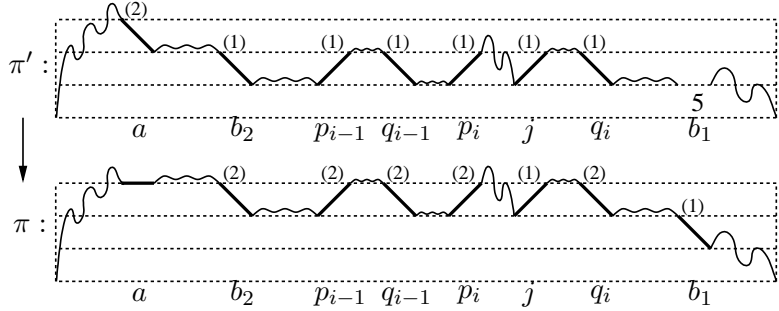}
\end{center}
\caption{\small Illustration for the proof of Proposition \ref{pro:d=2--Backward}.} \label{fig:proof-4-5}
\end{figure}

By Propositions \ref{pro:Left->Right} and \ref{pro:d=2--Backward}, we can recover the steps of the letters 4 (respectively, 5) of the word $\omega$ in the order from right to left. Over iterations, we construct the path $\phi_2^{-1}(\omega)\in\M^{(2)}_n$.
We have just established the bijection $\phi_2:\M^{(2)}_n\rightarrow\T^{(5)}_n$. For readers' reference, we list the one-to-one correspondence between the twenty-six SYTs with 5 cells and the paths in $\M_5^{(2)}$ at the end of this paper.

\section{The inductive stage}

Suppose a bijection $\phi_k:\M^{(k)}_n\rightarrow\T^{(2k+1)}$ has been established up to $k=d-1$. We shall establish a bijection $\phi_d:\M^{(d)}_n\rightarrow\T^{(2d+1)}_n$, $d\ge 3$.
\subsection{The map $\phi_d$}
Let $\pi=x_1x_2\cdots x_n\in\M^{(d)}_n$. As described in section 3,
we start from the first level step enclosed in a matching $(U^{(d)},D^{(d)})$-pair (respectively, start from the first $U^{(d)}$), say $x_a$, and determine a down-step sequence $(x_{b_d}, \dots,x_{b_1})=(D^{(d)},\dots,D^{(1)})$ such that $x_{b_d}$ is the first $D^{(d)}$ after $x_a$. For $1\le k\le d-1$, the step $x_{b_k}$ is located by a greedy procedure that searches for the nearest $D^{(k)}$. This is described in (i) and (ii) of the following algorithm.
Part (i) is for the greatest color $d$: if we encounter a critical $U^{(d)}$ before a $D^{(d-1)}$, we mark this $U^{(d)}$ and the first $D^{(d)}$ afterwards. Part (ii) is for the other colors: if we encounter a critical $U^{(k+1)}$ before a $D^{(k)}$, we mark this $U^{(k+1)}$ and go back to (ii) to search for the nearest $D^{(k+1)}$ afterwards. To describe the algorithm
conveniently, we use variables $P$ and $K$ as indices of steps and colors, respectively.

\noindent
{\bf Algorithm E.}

\begin{enumerate}
\item[(E1)] $\pi\in\overline{\M}^{(d)}_n$. Find the first level step enclosed in a matching $(U^{(d)},D^{(d)})$-pair, say $x_a$, and find the first $D^{(d)}$ after $x_a$, say $x_{b_d}$. Set $P=b_d$ and $K=d-1$. Go to (i).
\begin{itemize}
  \item[(i)] When $K=d-1$, we search for the nearest $D^{(d-1)}$ after $x_P$. If we encounter a critical $U^{(d)}$ before a $D^{(d-1)}$ then we mark this $U^{(d)}$ and the first $D^{(d)}$ afterwards. Set $x_P$ to this $D^{(d)}$ and repeat (i) until a $D^{(d-1)}$ is found.
      We mark this $D^{(d-1)}$,  set $x_P$ to this step, set $K=d-2$, and go to (ii).
  \item[(ii)] When $K\le d-2$, we search for the nearest $D^{(K)}$ after $x_P$. If we encounter a critical $U^{(K+1)}$ before a $D^{(K)}$ then we mark this $U^{(K+1)}$ and set $x_P$ to this step. Set $K=K+1$ and go to either of (i) and (ii).
      Otherwise, we locate the $D^{(K)}$. There are two cases. If $K=1$ then we are done and this step is the requested $x_{b_1}=D^{(1)}$.  If $K\ge 2$ then mark this $D^{(K)}$,  set $x_P$ to this step, set $K=K-1$ and go back to (ii).
\end{itemize}
When the above process stops, we obtain the sequence $(x_{b_d},\dots,x_{b_1})$ associated to $x_a$, where $x_{b_i}$ is the first marked $D^{(i)}$, for $1\le i\le d$. We form a new sequence $\pi'$ from $\pi$ as follows. Replace $(x_a,x_{b_d},\dots,x_{b_1})$ by $(D^{(d)},D^{(d-1)},\dots, D^{(1)},2d+1)$ and degrade every marked $U^{(k)}$ and $D^{(k)}$ to $U^{(k-1)}$ and $D^{(k-1)}$, respectively, for $2\le k\le d$. If $\pi'$ contains a level step enclosed in a matching $(U^{(d)},D^{(d)})$-pair then go back to (E1) and proceed to process $\pi'$, otherwise go to (E2).
\item[(E2)] $\pi\in\widehat{\M}^{(d)}_n$. Find the first $U^{(d)}$, say $x_a$, and find the first $D^{(d)}$ after $x_a$, say $x_{b_d}$. Set $P=b_d$ and $K=d-1$. Do the same procedures as (i) and (ii) of (E1) until we locate the requested step $x_{b_1}=D^{(1)}$. Then we form a new sequence $\pi'$ from $\pi$ as follows. Replace $(x_a,x_{b_d},\dots,x_{b_1})$ by $(L,D^{(d-1)},\dots,D^{(1)},2d)$ and degrade every marked $U^{(k)}$ and $D^{(k)}$ to $U^{(k-1)}$ and $D^{(k-1)}$, respectively, for $2\le k\le d$. If $\pi'$ contains a $U^{(d)}$ step then go back to (E2) and proceed to process $\pi'$, otherwise go to (E3).
\item[(E3)] $\pi\in\M^{(d-1)}_n$. By induction, we obtain the path $\phi_d(\pi)=\phi_{d-1}(\pi)$.
\end{enumerate}

\noindent
{\bf Remarks:} The current form of algorithm E is for $\phi_d$, $d\ge 3$. Note that the map $\phi_2$ can be integrated into $\phi_d$ by distinguishing $d=2$ and $d\ge 3$ in (i) of algorithm E.

\smallskip
We have the following observation for each iteration of the above algorithm.

\begin{lem} \label{lem:observation} In (E1) and (E2), the updated steps of $\pi$  are the step $x_a$, the down-step sequence $(x_{b_d}, \dots,x_{b_1})=(D^{(d)},\dots,D^{(1)})$ and a number of marked $(U^{(k)},D^{(k)})$-pairs in the subword between $x_{b_i}$ and $x_{b_{i-1}}$, for $2\le i\le k\le d$, such that every marked $(U^{(k+1)},D^{(k+1)})$-pair is nested in a marked $(U^{(k)},D^{(k)})$-pair, for $k\le d-1$.
\end{lem}

\begin{exa} \label{exa:d=3} {\rm
 Let us revisit the path $\pi=x_1\cdots x_{25}$ shown in Figure \ref{fig:critical-crucial}(a) in section 3. By (E1), we locate the steps $(x_a,x_{b_3})=(x_6,x_8)$. Searching for the nearest $D^{(2)}$ after $x_{b_3}$, we encounter a critical $U^{(3)}$ step, $x_9$. So mark $x_9$ and the first $D^{(3)}$ step, $x_{11}$, afterwards. Then we locate a $D^{(2)}$ step, $x_{12}=x_{b_2}$, and search for the nearest $D^{(1)}$ after $x_{b_2}$. On the way, we encounter a critical $U^{(2)}$ step, $x_{13}$, and go in a loop that searches for the nearest $D^{(2)}$ after $x_{13}$. Then we encounter a critical $U^{(3)}$ step, $x_{15}$. So mark $x_{15}$ and the first $D^{(3)}$ step, $x_{17}$, afterwards. Then we locate a $D^{(2)}$ step, $x_{19}$. Before we get to the $D^{(1)}$ step $x_{b_1}=x_{23}$, we encounter another critical $U^{(2)}$ step, $x_{20}$, and locate a $D^{(2)}$ step, $x_{21}$, afterwards.
We form a new sequence $\pi'$ from $\pi$ as follows. Replace $(x_a,x_{b_3},x_{b_2},x_{b_1})$ by $(D^{(3)},D^{(2)},D^{(1)},7)$. Moreover, degrade the marked pairs $(x_9,x_{11})$ and $(x_{15},x_{17})$  to  $(U^{(2)},D^{(2)})$, and degrade the marked pairs $(x_{13},x_{19})$ and $(x_{20},x_{21})$  to $(U^{(1)},D^{(1)})$. The resulting path $\pi'\in\widehat{\M}^{(3)}_{24}$ is shown in Figure \ref{fig:critical-crucial}(b) in section 3.

In Figure \ref{fig:critical-crucial}(b), by (E2), we locate the steps $(x_a,x_{b_3})=(x_5,x_6)$ of $\pi'$, and then we get to a $D^{(2)}$ step, $x_7=x_{b_2}$. Searching for the nearest $D^{(1)}$ after $x_{b_2}$, we encounter a critical $U^{(2)}$ step, $x_{10}$, and go in a loop that searches for the nearest $D^{(2)}$ after $x_{10}$. Then we locate a  $D^{(2)}$, step $x_{11}$, and  get to a $D^{(1)}$ step, $x_{12}=x_{b_1}$.
 We form a new sequence $\pi''$ from $\pi'$ as follows. Replace $(x_a,x_{b_3},x_{b_2},x_{b_1})$ by  $(L, D^{(2)},D^{(1)}, 6)$ and degrade the marked pair $(x_{10},x_{11})$ to  $(U^{(1)},D^{(1)})$. The resulting path $\pi''\in\widehat{\M}^{(3)}_{23}$ is shown in Figure \ref{fig:critical-up-step-25}(a). By (E2) once again, we get the steps $(x_{14},x_{22},x_{24},x_{25})=(U^{(3)},D^{(3)},D^{(2)},D^{(1)})$ replaced by $(L,D^{(2)},D^{(1)}, 6)$ and obtain the path shown in Figure \ref{fig:critical-up-step-25}(b).
}
\end{exa}

\begin{figure}[ht]
\begin{center}
\includegraphics[width=4.75in]{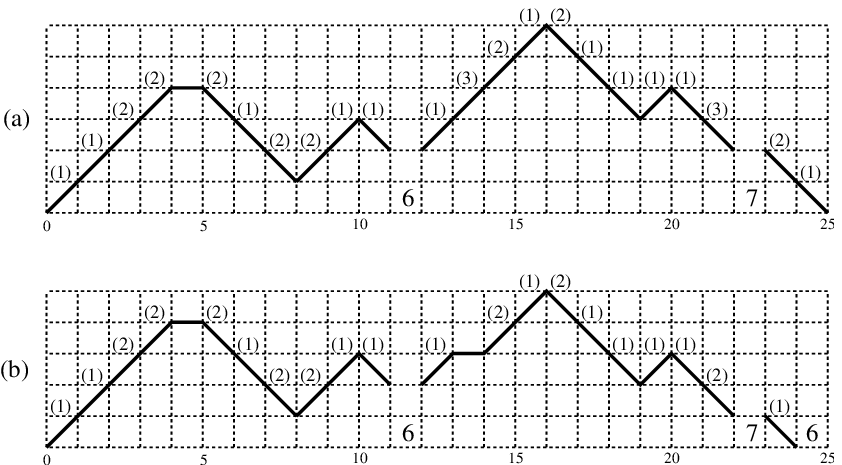}
\end{center}
\caption{\small A 3-Motzkin path in $\widehat{\M}^{(3)}_{23}$
and a 2-Motzkin path in $\overline{\M}^{(2)}_{22}$.} \label{fig:critical-up-step-25}
\end{figure}

We shall prove that the following properties of the word $\phi_d(\pi)$ constructed above.

\begin{pro} \label{pro:Left->Right(2d+1)} For a path $\pi\in\M^{(d)}_n$, the following results hold.
 \begin{enumerate}
 \item  We have $\phi_d(\pi)\in\T^{(2d+1)}$. Specifically, $\phi_d(\pi)\in\widetilde{\T}^{(2d)}_n$ if $\pi\in\widehat{\M}^{(d)}_n$, and $\phi_d(\pi)\in\widetilde{\T}^{(2d+1)}_n$ if $\pi\in\overline{\M}^{(d)}_n$.
 \item If $\phi_d(\pi)$ contains more than one letter $2d+1$  (respectively, $2d$) then upon iterations these letters appear in the order from left to right.
 \end{enumerate}
 \end{pro}

\begin{proof} (i)  Let $\pi=x_1\cdots x_n\in\overline{\M}^{(d)}_n$. Let $(x_{a},x_{b_d}, \dots,x_{b_1})$ be the sequence $(L,D^{(d)},\dots,D^{(1)})$ of $\pi$ located in (E1) and
let $\pi'$ be the path obtained from $\pi$ by (E1). We shall verify that $\pi'$ satisfies the conditions (M1) and (M2) for $d$-Motzkin words.

It is straightforward to check that $\pi'$ satisfies (M1). For (M2), it suffices to prove that for every prefix $\mu'(j)$ of $\pi'$ with $j<b_1$,
\begin{equation} \label{eqn:varify-(M2)}
 |\mu'(j)|_{U^{(1)}}-|\mu'(j)|_{D^{(1)}}\ge\cdots\ge |\mu'(j)|_{U^{(d)}}-|\mu'(j)|_{D^{(d)}}\ge 0.
\end{equation}

By induction hypothesis, for $j\le b_2$, we have $|\mu'(j)|_{U^{(2)}}-|\mu'(j)|_{D^{(2)}}\ge\cdots\ge |\mu'(j)|_{U^{(d)}}-|\mu'(j)|_{D^{(d)}}\ge 0$. (This is done by the same operation as in the map $\phi_{d-1}$ with colors shifted.) Moreover, we observe that $|\mu'(j)|_{U^{(1)}}-|\mu'(j)|_{D^{(1)}}\ge |\mu'(j)|_{U^{(2)}}-|\mu'(j)|_{D^{(2)}}$.

By Lemma \ref{lem:observation}, in the subword between $x_{b_2}$ and $x_{b_1}$ of $\pi$, every marked $(U^{(k+1)},D^{(k+1)})$-pair is nested in a marked $(U^{(k)},D^{(k)})$-pair, for $2\le k\le d-1$. These pairs are determined in a greedy way in (E1). By the same technique as in the proof of Proposition \ref{pro:one-iteration},
we can prove Eq.\,(\ref{eqn:varify-(M2)}) for $b_2< j< b_1$. Over iterations, we construct the word $\phi_d(\pi)\in\widetilde{\T}^{(2d+1)}_n$.

 (ii) For $\pi'$, let $(x_{a'},x_{b'_d}, \dots,x_{b'_1})$ be the located sequence $(L,D^{(d)},\dots,D^{(1)})$ in (E1).
 By induction hypothesis,  we have $b_2<b'_2$. We shall prove that $b_1<b'_1$. However, this can be proved by the same argument as in the proof of Proposition \ref{pro:Left->Right}.
 This proves the assertion for $\pi\in\overline{\M}^{(d)}_n$. For $\pi\in\widehat{\M}^{(d)}_n$, the proof is similar.
\end{proof}

\subsection{The map $\phi_d^{-1}$}

Let $\omega=x_1x_2\cdots x_n\in\T^{(2d+1)}_n$. The path corresponding to the subword of $\omega$ consisting of the letters $\{1,2,\dots,2d-1\}$ can be determined by the map $\phi_{d-1}^{-1}$. For the letters $2d$ (respectively, $2d+1$) of $\omega$, we shall recover their steps in the order from right to left, using the reverse operation of algorithm E. For convenience, in addition to variables $P$ and $K$, we use a symbol $\Lambda$ for the requested step $x_a$ in (E1) and (E2), which is either an $L$ or a $D^{(d)}$.

\noindent
{\bf Algorithm F.}

\begin{enumerate}
\item[(F1)] If $\omega$ contains no letters $2d$ then the path $\phi_d^{-1}(\omega)=\phi_{d-1}^{-1}(\omega)$ is obtained. Otherwise, go to (F2).
\item[(F2)] If $\omega$ contains a letter $2d$, then find the rightmost letter $2d$, say $x_c$. Let $\Lambda=L$. Set $P=c$ and $K=1$ and go to (i).
\begin{itemize}
 \item[(i)] When $K=1$, find the nearest $D^{(1)}$ to the left of $x_P$. Mark this $D^{(1)}$, set $x_P$ to this step, set $K=2$ and go to (ii).
 \item[(ii)] When $K\ge 2$, we search for the nearest $D^{(K)}$ to the left of $x_P$. If we encounter an exceeding $U^{(K-1)}$ before a $D^{(K)}$ then we mark this $U^{(K-1)}$ and set $x_P$ to this step. Set $K=K-1$ and go to either of (i) and (ii).
      Otherwise, we locate a $D^{(K)}$. Mark this $D^{(K)}$ and set $x_P$ to this step. There are two cases. If $K=d-1$ then go to (iii). If $K\le d-2$ then set $K=K+1$ and go back to (ii).
 \item[(iii)] Search for the nearest $\Lambda$ step to the left of $x_P$. If we encounter an exceeding $U^{(d-1)}$ before a $\Lambda$ step then we mark this $U^{(d-1)}$, set $x_P$ to this step and go to (ii), with $K=d-1$ currently.
 Otherwise, we locate a $\Lambda$ step, which is the requested $x_a$, and we are done.
\end{itemize}
     When the above process stops, we form a new sequence $\omega'$ from $\omega$ as follows. Replace $(x_a,x_c)$ by $(U^{(d)},D^{(1)})$ and upgrade every marked $U^{(k)}$ and $D^{(k)}$ between $x_a$ and $x_c$ to  $U^{(k+1)}$ and $D^{(k+1)}$, respectively, for $1\le k\le d-1$. If $\omega'$ contains no letter $2d$ then go to (F3); otherwise, go back to (F2) and proceed to process $\omega'$.
\item[(F3)] If $\omega$ contains a letter $2d+1$ then find the rightmost letter $2d+1$, say $x_c$. Let $\Lambda=D^{(d)}$. Do the same procedure as in (i)-(iii) of (F2) until we locate the requested step $x_a=\Lambda$. We end up with a number of marked $U^{(k)}$ and $D^{(k)}$ between $x_a$ and $x_c$. Then we form a new sequence $\omega'$ from $\omega$ as follows. Replace $(x_a,x_c)$ by $(L,D^{(1)})$ and upgrade every marked $U^{(k)}$ and $D^{(k)}$ between $x_a$ and $x_c$ to $U^{(k+1)}$ and $D^{(k+1)}$, respectively, for $1\le k\le d-1$. If $\omega'$ contains no letter $2d+1$ then we are done; otherwise, go back to (F3) and proceed to process $\omega'$.
\end{enumerate}

The following result can be proved by the technique used in the proof of Proposition \ref{pro:d=2--Backward}.

\begin{pro} \label{pro:d>=3--Backward} For a path $\pi\in\M^{(d)}_n$, let $\pi'$ be the path obtained from $\pi$  by (E1) (respectively, (E2)) in one iteration. Then $\pi$ can be recovered from $\pi'$  by (F3) (respectively, (F2)).
\end{pro}

\begin{proof} Let $\pi=x_1\cdots x_n\in\overline{\M}^{(d)}_n$. Let $(x_{a},x_{b_d}, \dots,x_{b_1})$ be the sequence $(L,D^{(d)},\dots,D^{(1)})$ of $\pi$ located in (E1) and
let $\pi'$ be the path obtained from $\pi$ by (E1).

 To recover $\pi$ from $\pi'$, we first restructure the subword from $x_{b_1}$ to $x_{b_2}$. By Lemma \ref{lem:observation}, in the subword between $x_{b_2}$ and $x_{b_1}$ of $\pi$, every marked $(U^{(k+1)},D^{(k+1)})$-pair is nested in a marked $(U^{(k)},D^{(k)})$-pair, for $2\le k\le d-1$. These pairs are determined in a greedy way.

Let $(x_{p_1},x_{q_1}),\dots,(x_{p_m},x_{q_m})$ be the marked $(U^{(k+1)},D^{(k+1)})$-pairs nested in a marked $(U^{(k)},D^{(k)})$-pair $(x_{s},x_{t})$ between $x_{b_2}$ and $x_{b_1}$ in $\pi$, for some $m$ and $k$.
We shall prove that these pairs can be recovered by (F3). Traverse $\pi'$ from right to left, we observe that
$x_{q_m}$ is the nearest $D^{(k)}$ to the left of $x_{t}$, each $x_{q_{i}}$ is the nearest $D^{(k)}$ to the left of $x_{p_{i+1}}$, $1\le i\le m-1$, and there are no $D^{(k)}$ between $x_{s}$ and $x_{p_1}$ since otherwise we would get to a $D^{(k)}$ before $x_{t}$ in $\pi$. It suffices to prove that $x_{p_{i}}$ can be located from right to left, starting from $x_{q_{i}}$ in $\pi'$.

If there are no marked $(U^{(k+2)},D^{(k+2)})$-pairs nested in $(x_{p_{i}},x_{q_{i}})$ in $\pi$, then by an argument similar to the proof of Proposition \ref{pro:d=2--Backward}, $x_{p_{i}}$ is the nearest exceeding $U^{(k)}$ to the left of $x_{q_{i}}$. Otherwise, there are marked $(U^{(k+2)},D^{(k+2)})$-pairs nested in $(x_{p_{i}},x_{q_{i}})$ in $\pi$. Then by the same argument as above, we prove inductively that these pairs can be recovered and hence $x_{p_{i}}$ is located. This recovers the subword from $x_{b_1}$ to $x_{b_2}$.

By induction hypothesis and by (ii) and (iii) of algorithm F, we can recover the sequence $x_{b_2},\dots,x_{b_d},x_a$ from right to left, along with the marked $(U^{(k)},D^{(k)})$-pairs between $x_{b_i}$ and $x_{b_{i+1}}$, for $2\le i< k\le d$. (This is done by the same operation as in the map $\phi_{d-1}$ with colors shifted.)
This proves the assertion for $\pi\in\overline{\M}^{(d)}_n$. For the other case $\pi\in\widehat{\M}^{(d)}_n$, the proof is similar.
\end{proof}

By Propositions \ref{pro:Left->Right(2d+1)} and \ref{pro:d>=3--Backward}, we can recover the steps of the letters $2d$ (respectively, $2d+1$) of the word $\omega$ in the order from right to left. Over iterations, we construct the path $\phi_d^{-1}(\omega)\in\M^{(d)}_n$. This proves Theorem \ref{thm:main-bijection}.

\smallskip
From the construction of the bijection $\phi_d$, we have the
following relation regarding the level steps of $d$-Motzkin paths
and the odd columns of SYTs.

\begin{cor} \label{cor:5-5} The number of level steps of a path $\pi\in\M^{(d)}_n$ equals the number of odd columns of the SYT $\phi_d(\pi)\in\T^{(2d+1)}_n$.
\end{cor}

\begin{proof} Let $T=\phi_d(\pi)$. Let $\kappa(\pi)$ (respectively, $\kappa(T)$) be the number of level steps of $\pi$ (respectively, odd columns of $T$).
Let $\pi'$ be the path obtained from $\pi$ in one iteration of the algorithm E. Let $T'=\phi_d(\pi')$. Note that the shape of $T'$ differs from that of $T$ by one cell.

By (E1), we convert a level step of $\pi$, $x_a$, into a down step and replace a $D^{(1)}$ step by an odd entry. The other updated steps are modified with color only  (no extra level steps are created).  So a level step of $\pi$ is associated with the last cell of an odd column of $T$. Hence $\kappa(\pi')=\kappa(\pi)-1$ and $\kappa(T')=\kappa(T)-1$.
By (E2), we convert an up step of $\pi$, $x_a$, into a level step and replace a $D^{(1)}$ step by an even entry. Likewise, the other updated steps are modified with color only. Hence $\kappa(\pi')=\kappa(\pi)+1$ and $\kappa(T')=\kappa(T)+1$.

These observations also hold for the algorithms C and A. Over iterations, we eventually arrive at a path, say $\pi''$, consisting of $\ell$ level steps on the $x$-axis, for some $\ell>0$. By (A3), the corresponding SYT $T''=\phi_1(\pi'')$ consists of $\ell$ columns with one cell. Hence, we have
\[
\kappa(\pi)-\kappa(T)=\kappa(\pi')-\kappa(T')=\cdots=\kappa(\pi'')-\kappa(T'')=0,
\]
as required.
\end{proof}

As shown in Figure \ref{fig:map-list}, the bijection $\phi_2$ carries a path  $\pi\in\M_5^{(2)}$ with $k$ level steps to an SYT $\phi_2(\pi)\in\T_5^{(5)}$ with $k$ odd columns.

\begin{figure}[ht]
\begin{center}
\includegraphics[width=6in]{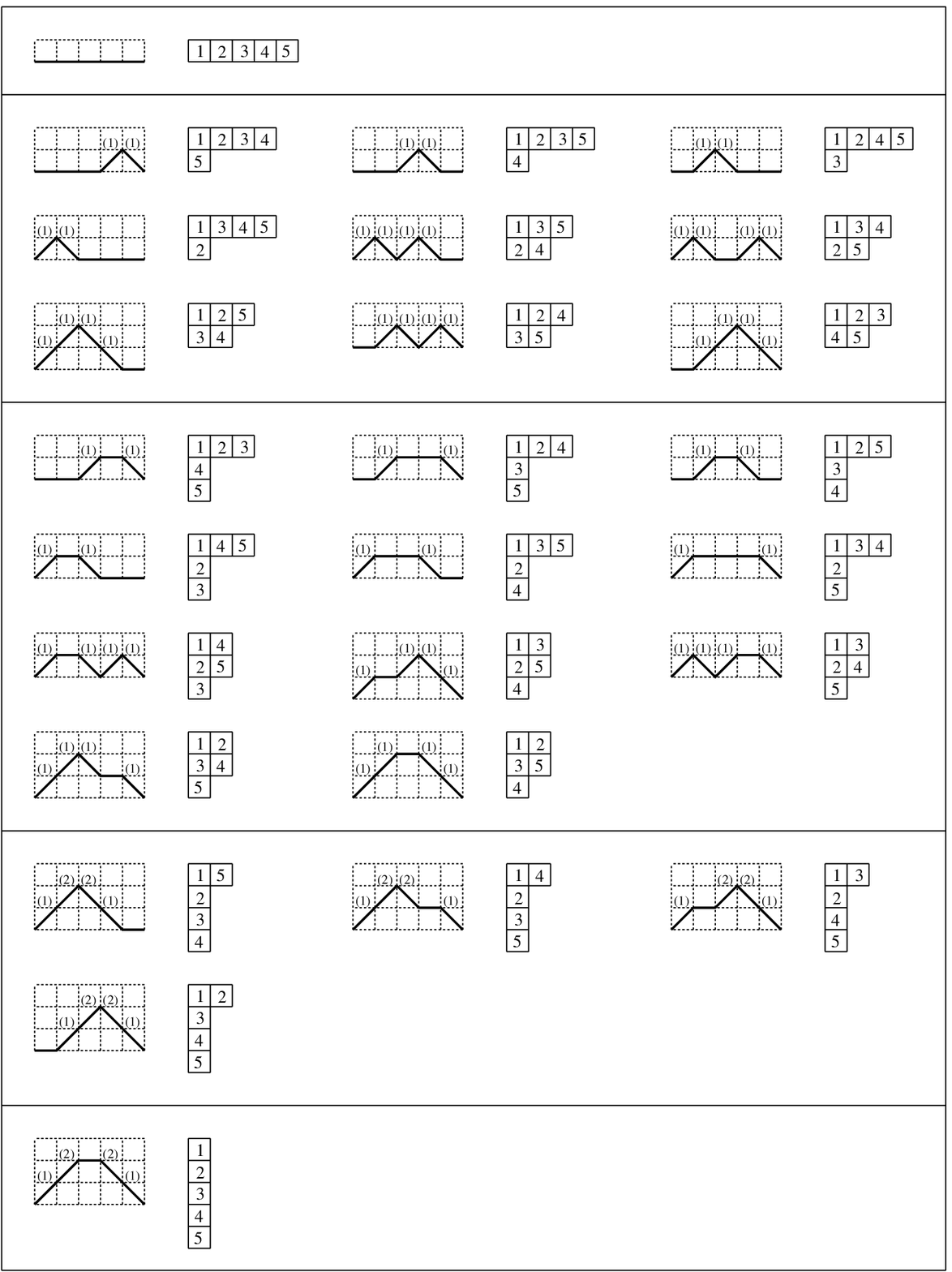}
\end{center}
\caption{\small The bijection $\phi_2:\M_5^{(2)}\rightarrow\T_5^{(5)}$.} \label{fig:map-list}
\end{figure}

\section{Concluding remarks}

In this paper, the notion of colored Motzkin paths is introduced and
a bijection with standard Young tableaux of bounded height is
established. Moreover, the bijection reveals a link between the two sets $\T_n^{(2d+1)}$ and $\T_n^{(2d)}$ via a property of level steps of the paths.

Using the bijection, we can investigate SYTs and other related objects in
terms of lattice paths. Some interesting work is in progress. For example, it is also known that an SYT $T$ with $n$ cells corresponds to an involution $\sigma$ on the set $\{1,\dots,n\}$ by the
RSK algorithm, with the number of odd columns of $T$ equal to
the number of fixed points of $\sigma$ (e.g.,
see \cite[Exercise 7, p.135]{Sagan}). It is interesting to establish a direct connection between colored Motzkin paths and involutions that carries the statistic level step to fixed point.

We hope the notion of colored Motzkin paths can shed some new light
on the study of standard Young tableaux.

\section*{Acknowledgements.}
The authors thank the referees for carefully reading the manuscript
and providing helpful suggestions. The authors also thank Ting-Yuan
Cheng for stimulating discussions.

\end{document}